\documentclass[reqno]{amsart}
\usepackage{hyperref}
\usepackage{amssymb}

\allowdisplaybreaks

\theoremstyle{plain}
\newtheorem{theorem}{Theorem}[section]

\newtheorem{lemma}[theorem]{Lemma}

\theoremstyle{definition}
\newtheorem{definition}[theorem]{Definition}

\theoremstyle{remark}
\newtheorem{remark}[theorem]{Remark}
\numberwithin{equation}{section}

\begin{document}

\title[Blow-up problem for semilinear heat equation]
{Blow-up problem for semilinear heat equation
with nonlinear nonlocal Neumann boundary condition}

\author[A. Gladkov]{Alexander Gladkov}
\address{Alexander Gladkov \\ Department of Mechanics and Mathematics
\\ Belarusian State University \\ Nezavisimosti Avenue 4, 220030
Minsk, Belarus} \email{gladkoval@mail.ru}



\begin{abstract}
In this paper, we consider a semilinear parabolic equation with
nonlinear nonlocal Neumann boundary condition and nonnegative
initial datum. We first prove global existence results. We then
give some criteria on this problem which determine whether the
solutions blow up in finite time for large or for all nontrivial
initial data. Finally, we show that under certain conditions
blow-up occurs only on the boundary.
\end{abstract}

\maketitle

\section{Introduction}\label{in}
In this paper we consider the initial boundary value problem for
the following semilinear parabolic equation
\begin{equation}\label{v:u}
    u_t=\Delta u-c(x,t)u^p,\;x\in\Omega,\;t>0,
\end{equation}
with nonlinear nonlocal boundary condition
\begin{equation}\label{v:g}
\frac{\partial
u(x,t)}{\partial\nu}=\int_{\Omega}{k(x,y,t)u^l(y,t)}\,dy,\;x\in\partial\Omega,\;
t \geq 0,
\end{equation}
and initial datum
\begin{equation}\label{v:n}
    u(x,0)=u_{0}(x),\; x\in\Omega,
\end{equation}
where $p>0,\,l>0$, $\Omega$ is a bounded domain in $\mathbb{R}^n$
for $n\geq1$ with smooth boundary $\partial\Omega$, $\nu$ is unit
outward normal on $\partial\Omega.$

Throughout this paper we suppose that the functions
$c(x,t),\;k(x,y,t)$ and $u_0(x)$ satisfy the following conditions:
\begin{equation*}
c(x,t)\in
C^{\alpha}_{loc}(\overline{\Omega}\times[0,+\infty)),\;0<\alpha<1,\;c(x,t)\geq0;
\end{equation*}
\begin{equation*}
k(x, y, t)\in
C(\partial\Omega\times\overline{\Omega}\times[0,+\infty)),\;k(x,y,t)\geq0;
\end{equation*}
\begin{equation*}
u_0(x)\in C^1(\overline{\Omega}),\;u_0(x)\geq0 \textrm{ in }
\Omega,  \;\frac{\partial u_0(x)}{\partial\nu}=\int_{\Omega}{k(x,
y,0)u_0^l(y)}\,dy \textrm{ on } \partial\Omega.
\end{equation*}
Many authors have studied blow-up problem for parabolic equations
and systems with nonlocal boundary conditions (see, for
example,~\cite{CL}--\cite{ZK}
and the references therein). In particular, the initial boundary
value problem for equation~(\ref{v:u}) with nonlinear nonlocal
boundary condition
\begin{equation*}
u(x,t)=\int_{\Omega}k(x,y,t)u^l(y,t)\,dy,\;x\in\partial\Omega,\;t>0,
\end{equation*}
was considered for $c(x,t) \leq 0$ and $c(x,t) \geq 0$
in~\cite{GladkovKim} and~\cite{GladkovGuedda}, respectively. The
problem~(\ref{v:u})--(\ref{v:n}) with $c(x,t) \leq 0$ was
investigated in~\cite{Gladkov_Kavitova2} and closed problem was
analyzed in~\cite{MV}.

Local existence theorem, comparison and uniqueness results for
problem~(\ref{v:u})--(\ref{v:n}) have been established
in~\cite{Gladkov}.

In this paper we obtain necessary and sufficient conditions for
the existence of global solutions as well as for a blow-up in
finite time of solutions for problem~(\ref{v:u})--(\ref{v:n}). Our
global existence and blow-up results depend on the behavior of the
functions $c(x,t)$ and $k(x,y,t)$ as $t \to \infty.$

This paper is organized as follows. The global existence theorem
for any initial data and blow-up in finite time of solutions for
large initial data are proved in section 2. In section 3 we
present finite time blow-up of all nontrivial solutions as well as
the existence of global solutions for small initial data. Finally,
in section 4 we show that under certain conditions blow-up occurs
only on the boundary.

\section{Global existence}\label{gl}
\setcounter{equation}{0} \setcounter{theorem}{0}

Let $Q_T=\Omega\times(0,T),\;S_T=\partial\Omega\times(0,T)$,
$\Gamma_T=S_T\cup\overline\Omega\times\{0\}$, $T>0$.
\begin{definition}\label{v:sup}
We say that a nonnegative function $u(x,t)\in C^{2,1}(Q_T)\cap
C^{1,0}(Q_T\cup\Gamma_T)$ is a supersolution
of~(\ref{v:u})--(\ref{v:n}) in $Q_{T}$ if
        \begin{equation}\label{v:sup^u}
        u_{t}\geq\Delta u-c(x, t)u^{p},\;(x,t)\in Q_T,
        \end{equation}
        \begin{equation}\label{v:sup^g}
        \frac{\partial u(x,t)}{\partial\nu}\geq\int_{\Omega}{k(x, y, t)u^l(y, t) }\,dy,
        \; x \in \partial \Omega,\; 0 \leq t < T,
        \end{equation}
        \begin{equation}\label{v:sup^n}
            u(x,0)\geq u_{0}(x),\; x\in\Omega,
        \end{equation}
and $u(x,t)\in C^{2,1}(Q_T)\cap C^{1,0}(Q_T\cup\Gamma_T)$ is a
subsolution of~(\ref{v:u})--(\ref{v:n}) in $Q_{T}$ if $u\geq0$ and
it satisfies~(\ref{v:sup^u})--(\ref{v:sup^n}) in the reverse
order. We say that $u(x,t)$ is a solution of
problem~(\ref{v:u})--(\ref{v:n}) in $Q_T$ if $u(x,t)$ is both a
subsolution and a supersolution of~(\ref{v:u})--(\ref{v:n}) in
     $Q_{T}$.
\end{definition}

To prove the main results we use the positiveness of a solution
and the comparison principle which have been proved
in~\cite{Gladkov}.
\begin{theorem}\label{p:theorem:positive}
Let $u_0$ is a nontrivial function in $\Omega,$  $p \geq 1$ or
$c(x,t) \equiv 0.$ Suppose $u$ is a solution of
(\ref{v:u})--(\ref{v:n}) in $Q_T.$ Then $u>0$ in ${Q}_T \cup S_T.$
\end{theorem}
\begin{theorem}\label{p:theorem:comp-prins}
Let $\overline{u}$ and $\underline{u}$ be a
 supersolution and a  subsolution of problem
(\ref{v:u})--(\ref{v:n}) in $Q_T,$ respectively. Suppose that
$\underline{u}(x,t)> 0$ or $\overline{u}(x,t) > 0$ in ${Q}_T\cup
\Gamma_T$ if $l < 1$. Then $ \overline{u}(x,t) \geq
\underline{u}(x,t) $ in ${Q}_T\cup \Gamma_T.$
\end{theorem}
The proof of a global existence result relies on the continuation
principle and the construction of a supersolution. We suppose that
\begin{equation}\label{E0}
    c(x,t) > 0, \, x \in \overline{\Omega}, \, t \geq 0.
\end{equation}
\begin{theorem}\label{global}
Let $l \leq 1$ or $1 < l < p$ and (\ref{E0}) hold. Then problem
(\ref{v:u})--(\ref{v:n}) has a global solution for any initial
datum.
\end{theorem}
\begin{proof}
In order to prove the existence of global solutions we construct a
suitable explicit supersolution of~(\ref{v:u})--(\ref{v:n}) in
$Q_T$ for any positive $T.$ Suppose at first that $l \leq 1.$
Since $k(x,y,t)$ is a continuous function there exists a constant
$K>0$ such that
\begin{equation}\label{K}
k(x,y,t)\leq K
\end{equation}
in $\partial\Omega\times Q_T.$ Let $\lambda_1$ be the first
eigenvalue of the following problem
\begin{equation*}
    \begin{cases}
        \Delta\varphi+\lambda\varphi=0,\;x\in\Omega,\\
        \varphi(x)=0,\;x\in\partial\Omega,
    \end{cases}
\end{equation*}
and $\varphi(x)$ be the corresponding eigenfunction with
$\sup\limits_{\Omega}\varphi(x)=1$. It is well known
$\varphi(x)>0$ in $\Omega$ and $\max\limits_{\partial\Omega}
\partial\varphi(x)/\partial\nu < 0.$

Now we show that
$$
\overline u (x,t) = \frac{C\exp (\mu t)}{a \varphi (x) + 1}
$$
is a supersolution of~(\ref{v:u})--(\ref{v:n}) in $Q_T,$ where
constants $C,\mu$ and $a$ are chosen to satisfy the following
inequalities:
\begin{equation*}
a\geq \max \left\{ K \int_\Omega \frac{dy}{(\varphi (y) + 1)^l}
\max_{\partial \Omega} \left(-\frac{\partial \varphi}{\partial
\nu} \right)^{-1}, 1 \right\},
\end{equation*}
\begin{equation*}
C \geq \max \{ \sup_\Omega (a \varphi (x) + 1) u_0 (x), 1 \},
\quad \mu \geq \lambda_1 + 2 a^2 \sup_\Omega \frac{|\nabla \varphi
|^2 } {(a \varphi (x) + 1)^2}.
\end{equation*}
Indeed, it is easy to check that
\begin{equation}\label{E1}
\overline u_t - \Delta\overline u + c(x,t)\overline u^p \geq
\left( \mu - \frac{a\lambda_1\varphi}{(a \varphi (x) + 1)^2} - 2
a^2 \sup_\Omega \frac{|\nabla \varphi |^2 } {(a \varphi (x) +
1)^2} \right) \overline u  \geq 0
\end{equation}
for $(x,t) \in Q_T,$
\begin{eqnarray}\label{E2}
\frac{\partial \overline u}{\partial\nu} &=& a C \exp (\mu t)
\left(-\frac{\partial \varphi}{\partial \nu} \right) \geq  K C^l
\exp (l\mu t) \int_\Omega \frac{dy}{(\varphi (y) + 1)^l} \nonumber \\
&\geq& \int_{\Omega} k(x,y,t)\overline u^l(y,t) \,dy
\end{eqnarray}
for $(x,t) \in S_T$ and
\begin{equation}\label{E3}
\overline u(x,0)\geq u_0(x)
\end{equation}
for $x \in \Omega.$ It follows from (\ref{E1})--(\ref{E3}) that
problem (\ref{v:u})--(\ref{v:n}) has a global solution for any
initial datum.

Suppose now that $ 1<l<p$ and (\ref{E0}) holds. By (\ref{E0}) we
have $c(x,t)\geq \underline{c}$ in $Q_T,$ where $\underline{c}$ is
some positive constant.

To construct a  supersolution we use the change of variables in a
neighborhood of $\partial \Omega$ as in \cite{CPE}. Let $\overline
x$ be a point in $\partial \Omega.$ We denote by $\widehat{n}
(\overline x)$ the inner unit normal to $\partial \Omega$ at the
point $\overline x.$ Since $\partial \Omega$ is smooth it is well
known that there exists $\delta >0$ such that the mapping $\psi
:\partial \Omega \times [0,\delta] \to \mathbb{R}^n$ given by
$\psi (\overline x,s)=\overline x +s\widehat{n} (\overline x)$
defines new coordinates ($\overline x,s)$ in a neighborhood of
$\partial \Omega$ in $\overline\Omega.$ A straightforward
computation shows that, in these coordinates, $\Delta$ applied to
a function $g(\overline x,s)=g(s),$ which is independent of the
variable $\overline x,$ evaluated at a point $(\overline x,s)$ is
given by
\begin{equation}\label{E:1.102}
\Delta g(\overline x,s) = \frac{\partial^2g}{\partial s^2}
(\overline x,s) - \sum_{j=1}^{n-1} \frac{H_j (\overline x)}{1-s
H_j (\overline x)}\frac{\partial g}{\partial s} (\overline x,s),
\end{equation}
where $H_j (\overline x)$ for $j=1,...,n-1,$ denotes the principal
curvatures of $\partial\Omega$ at $\overline x.$

For points in $Q_{\delta,T}=\partial \Omega \times [0,
\delta]\times [0,T]$ of coordinates $(\overline x,s,t)$ define
\begin{equation}\label{E:4.1}
\overline u (\overline x,s,t)= \left[ (\alpha s +
\varepsilon)^{-\gamma} -
\omega^{-\gamma}\right]_+^\frac{\beta}{\gamma} + A,
\end{equation}
where $\alpha >0, \,$ $0<\varepsilon<\omega<\alpha\delta, \,$
$\max\{1/l,2/(p-1)\} < \beta < 2/(l-1),\,$ $0<\gamma<\beta/2,\,$
$A \ge \sup_{\Omega} u_0(x),\,$ $\sigma_+=\max\{\sigma,0\}.$ For
points in $\overline{Q_T}\setminus Q_{\delta,T}$ we put $\overline
u (\overline x,s,t)= A.$ It has been showed in
\cite{GladkovGuedda} that
\begin{equation*}
\overline u_t - \Delta\overline u + c(x,t)\overline u^p \geq 0,
\quad (x,t) \in Q_T
\end{equation*}
for small $\varepsilon$ and large $A.$

Now we show that
\begin{equation}\label{E:4.6}
\frac{\partial \overline u}{\partial\nu} (\overline x,0,t) \geq
\int_{\Omega} k(x,y,t) \overline u^l(\overline x,s,t) \, dy, \quad
(x,t) \in Q_T
\end{equation}
for a suitable choice of $\varepsilon.$ To estimate the integral
$I$ in the right hand side of (\ref{E:4.6}) we shall use the
change of variables in a neighborhood of $\partial\Omega.$ Let
$$
 \overline J= \sup_{0< s< \delta} \int_{\partial\Omega}
|J(\overline y,s)| \, d\overline y,
$$
where $J(\overline y,s)$ is Jacobian of the change of variables.
Then we have
\begin{eqnarray*}
I & \leq & 2^{l-1} K \int_{\Omega} \left[ (\alpha s +
\varepsilon)^{-\gamma} - \omega^{-\gamma}\right]_+^\frac{\beta
l}{\gamma} \, dy + 2^{l-1} K A^l |\Omega|\\
&\leq & 2^{l-1} K  \overline J \int_{0}^{(\omega -
\varepsilon)/\alpha} \left[ (\alpha s + \varepsilon)^{-\gamma} -
\omega^{-\gamma}\right]^\frac{\beta
l}{\gamma} \, ds + 2^{l-1} K A^l |\Omega|\\
&\leq & \frac{2^{l-1} K  \overline J}{\alpha(\beta l-1)} \left[
\varepsilon^{-(\beta l-1)} - \omega^{-(\beta l-1)}\right] +
2^{l-1} K A^l |\Omega|,
\end{eqnarray*}
where $K$ was defined in (\ref{K}), $|\Omega|$ is Lebesque measure
of $\Omega.$ On the other hand, since
$$
\frac{\partial \overline u}{\partial\nu} (\overline x,0,t) = -
\frac{\partial \overline u}{\partial s} (\overline x,0,t) = \alpha
\beta \varepsilon^{-\gamma -1} \left[ \varepsilon^{-\gamma} -
\omega^{-\gamma}\right]_+^\frac{\beta-\gamma}{\gamma},
$$
the inequality (\ref{E:4.6}) holds if $\varepsilon$ is small
enough and hence by Theorem~\ref{p:theorem:comp-prins} we get
$$
 u(x,t) \leq \overline u (\overline x,s,t) \,\,\,  \textrm{in} \,\,\,
 \overline{Q}_T.
$$
\end{proof}
\begin{remark}\label{Rem6}
Let
$$
\underline\lambda=\frac{\inf_{\Omega\times(0,+\infty)}c(x,t)}
{\sup_{\partial\Omega\times\Omega\times(0,+\infty)}k(x,y,t)}.
$$
Note that under $\beta = 2/(l-1)$ and a suitable choice of
$\alpha$ in (\ref{E:4.1}) the same proof holds if $l=p>1$ and
$\underline\lambda$ is large enough and consequently a solution of
problem (\ref{v:u})--(\ref{v:n}) is global.
\end{remark}
Now we shall prove finite time blow-up result. We suppose that
\begin{equation}\label{E:4.7}
k(x,y,t_0)>0, \,\,\,  x \in \partial\Omega,\, y\in
\partial\Omega.
\end{equation}
\begin{theorem}\label{Th10} Let $l>\max\{1, p\}$ and (\ref{E:4.7})
hold with $t_0 \geq 0$ if $p \leq 1$ and with $t_0 = 0$ if $p >
1.$ Then there exist solutions of (\ref{v:u})--(\ref{v:n}) with
finite time blow-up.
\end{theorem}
\begin{proof}
At first we suppose that $p\leq 1,\,$ $l>1$ and (\ref{E:4.7})
holds with $t_0 \geq 0.$ To prove the theorem we construct a
subsolution of an auxiliary problem which blows up in finite time.
First of all we get a lower bound for solutions of
(\ref{v:u})--(\ref{v:n}) with positive initial data. We denote
\begin{equation}\label{E:5.1c}
\overline c (t)=\sup\limits_{\Omega} c(x,t).
\end{equation}
It is not difficult to check that
\begin{eqnarray*}
w(t) = \left\{ \begin{array}{ll} \left[ A^{1-p}  - (1-p) \int _0^t
\overline c (\tau)\, d\tau \right]^{1/(1-p)}
\,\,\,&\textrm{for} \,\,\, 0 <p<1,\\
A  \exp \left[ - \int _0^t \overline c (\tau)\, d\tau \right]
\,\,\,&\textrm{for} \,\,\, p=1
\end{array} \right.
\end{eqnarray*}
is a subsolution of (\ref{v:u})--(\ref{v:n}) in $Q_T$ for any
$T>0$ if
\begin{equation}\label{E:4.12a}
 u_0(x) \geq A>0.
\end{equation}
Then by Theorem~\ref{p:theorem:comp-prins} we have
\begin{equation}\label{E:4.12}
u(x,t) \geq  w(t)  \,\,\, \textrm{for} \,\,\, x \in \overline
\Omega \,\,\,\textrm{and} \,\,\,  t \geq 0.
\end{equation}
Consider the change of variables in a neighborhood of
$\partial\Omega$ as in Theorem~\ref{global}. Set $\Omega_\gamma =
\{ (\overline x,s): \overline x \in \partial\Omega, 0<s<\gamma
\}.$ By (\ref{E:4.7}) we have
\begin{equation}\label{E:4.7a}
k(x,y,t) \geq k_1, \,\,\,  x \in \partial\Omega,\, y \in
\Omega_\gamma, \, t_0 < t < t_1
\end{equation}
for some positive $k_1,\,$ $\gamma$ and $t_1 > t_0.$

Let us consider the following initial boundary value problem:
\begin{eqnarray}\label{E:4.13}
\left\{
\begin{array}{ll}
v_{t}=\Delta v - c(x,t) v^p  \,\,\, \,\,\,& \textrm{for} \,\,\,
x \in \Omega_\gamma, \,\,\, t_0<t<t_2,   \\
\frac{\partial v(x,t)}{\partial \nu} = \int_{\Omega_\gamma}
k(x,y,t) v^l (y,t) \, dy \,\,\,\,\,\,\,\,& \textrm{for}\,\,\, x
\in
\partial\Omega, \,\,\,  t_0<t<t_2, \\
v(x,t)= u(x,t) \,\,\,\,\,\,& \textrm{for}\,\,\, x \in
\partial\Omega_\gamma\setminus\partial\Omega, \,\,\, t_0<t<t_2, \\
v(x,t_0)= u (x,t_0) \,\,\,\,\,\,& \textrm{for}\,\,\, x \in
\Omega_\gamma,
\end{array} \right.
\end{eqnarray}
where $\nu$ is unit outward normal on $\partial\Omega,$ $u(x,t)$
is a solution of (\ref{v:u})--(\ref{v:n}), $t_2 \in (t_0, t_1)$
and will be chosen later. We can define the notions of a
supersolution and a subsolution of (\ref{E:4.13}) in a similar way
as in Definition~\ref{v:sup}. We shall use a comparison principle
for a subsolution and a supersolution of (\ref{E:4.13}) which can
be proved analogously to Theorem~\ref{p:theorem:comp-prins}. It is
easy to see that $u(x,t)$ is a supersolution of (\ref{E:4.13}) in
$Q(\gamma,t_0,t_2) =\Omega_\gamma \times (t_0, t_2).$

We define
\begin{equation}\label{E:4.14}
\psi(s,t) = (t_2 + s-t)^{-\sigma},
\end{equation}
where $\sigma >2/(l-1)$ and show that $\psi(s,t)$ is a subsolution
of (\ref{E:4.13}) in $Q(\gamma,t_0,t_2)$ under suitable choice of
$t_2$ and $\gamma.$ It is obvious, $\psi(0,t) \to \infty$ as $t
\to t_2.$

For $0 < s <\gamma$ and small $\gamma$ we have
\begin{equation}\label{E:4.14a}
\left|\sum_{j=1}^{n-1} \frac{H_j (\overline x)}{1-s H_j (\overline
x)} \right| \leq C.
\end{equation}
Using (\ref{E:1.102}), (\ref{E:4.14}), (\ref{E:4.14a}) we find
that
\begin{eqnarray*}
-\psi_t + \Delta \psi - c(x,t) \psi^p &\geq& (t_2  +
s-t)^{-\sigma-2} \left\{ \sigma (\sigma + 1) - \sigma (C + 1) (t_2
- t_0 + \gamma) \right.\\
 &-& \sup_{(t_0, t_2)} \overline c (t) (t_2 - t_0 +
\gamma)^{\sigma + 2 - \sigma p} \left. \right\} \geq 0
\end{eqnarray*}
in $Q(\gamma,t_0,t_2)$ if we take $\gamma$ and $t_2 -t_0$ small
enough. Now we prove that
\begin{equation*}\label{E:4.15}
\frac{\partial \psi}{\partial \nu} (0,t) \leq \int_{\Omega_\gamma}
k(x,y,t) \psi^l (s,t) \, dy \,\,\, \textrm{for}\,\,\, x \in
\partial\Omega, \, t_0<t<t_2.
\end{equation*}
To do this, we use the change of variables in a neighborhood of
$\partial\Omega.$ Let
$$
\underline J= \inf_{0< s< \gamma} \int_{\partial\Omega}
|J(\overline y,s)| \, d\overline y,
$$
where $J(\overline y,s)$ is Jacobian of the change of variables.
By virtue of (\ref{E:4.7a}), (\ref{E:4.14}) we have
\begin{eqnarray*}
\frac{\partial \psi}{\partial \nu} (0,t) &-& \int_{\Omega_\gamma}
k(x,y,t) \psi^l (s,t) \, dy  \\
&\leq& \sigma (t_2 - t)^{-\sigma -1} - k_1 \underline J
\int_0^\gamma (t_2 + s - t)^{-\sigma l} ds  \\
&\leq& \sigma (t_2 - t)^{-\sigma -1} - k_1 \underline J \frac{(t_2
- t)^{-\sigma l + 1}}{\sigma l - 1} \left[ 1 - \left(
\frac{t_2}{t_2 + \gamma} \right)^{\sigma l - 1} \right] \leq 0
\end{eqnarray*}
for $x \in \partial\Omega, \, t_0<t<t_2$ and small enough $t_2 -
t_0.$

We suppose now that
\begin{equation}\label{E:4.15a}
\gamma < t_2 - t_0,
\end{equation}
\begin{equation}\label{E:4.15b}
A \geq \left[ (1-p) \int _0^{t_1} \overline c (\tau)\, d\tau +
\gamma^{-\sigma (1-p)} \right]^{1/(1-p)} \,\,\, \textrm{for}
\,\,\, 0 <p<1,
\end{equation}
\begin{equation}\label{E:4.15c}
A \geq \gamma^{-\sigma} \exp \left[ \int _0^{t_1} \overline c
(\tau) \right] \,\,\, \textrm{for} \,\,\, p = 1.
\end{equation}
Due to (\ref{E:4.12a}), (\ref{E:4.12}), (\ref{E:4.15a}) --
(\ref{E:4.15c}) we have
$$
\psi(s,t) \leq  u(x,t) \,\,\, \textrm{for} \,\,\, x \in
\Omega_\gamma,\, t=t_0 \,\,\,\textrm{and} \,\,\, x \in
\partial\Omega_\gamma\setminus\partial\Omega, \,\,\, t_0\leq t\leq t_2.
$$

Comparing $u(x,t)$ and $\psi(s,t)$ in $Q(\gamma,t_0,t_2)$
 we prove the theorem for $p\leq 1,\,$ $l>1.$

Let  $l>p>1$ and (\ref{E:4.7}) hold with $t_0 = 0.$ We denote $c_1
= \sup_{Q_{t_1}} c (x,t)$ and suppose that
\begin{equation*}
\max \left\{ \frac {1}{p-1},\frac{2}{l-1} \right\} < \sigma <
\frac{2}{p-1}, \, u_0(x) \geq \max \left\{[t_2 (p-1) c_1
]^{-\frac{1}{p-1}}, t_2^{-\sigma} \right\},
\end{equation*}
where $t_2 \in (0, t_1)$ and will be chosen later. It is not
difficult to check that
$$
w(t) = [(p-1) c_1 (t+t_2)]^{-\frac{1}{p-1}}
$$
is a subsolution of (\ref{v:u})--(\ref{v:n}) in $Q_{t_2}.$ Then by
Theorem~\ref{p:theorem:comp-prins} we have
\begin{equation*}
w(t) \leq  u(x,t)  \,\,\, \textrm{for} \,\,\, x \in \overline
\Omega \,\,\,\textrm{and} \,\,\, 0 \leq t \leq t_2.
\end{equation*}
In the same way as in a previous case we can show that $\psi(s,t)$
is a subsolution of (\ref{E:4.13}) in $Q(\gamma,t_0,t_2)$ with
$t_0 = 0$ for small values of $\gamma$ and
\begin{equation*}
t_2 \leq \min \left\{ t_1, \frac{\gamma^{\sigma (p-1)}}{2(p-1)
c_1} \right\}.
\end{equation*}
\end{proof}
\begin{remark}\label{Rem7}
We put
$$
\overline\lambda=\frac{\sup_{\partial\Omega}c(x,0)}
{\inf_{\partial\Omega\times\partial\Omega}k(x,y,0)}
$$
and consider
\begin{equation}\label{E:4.15d}
\psi(s,t) = (t_2 + \omega s-t)^{-2/(p-1)}, \, \omega >0
\end{equation}
instead of (\ref{E:4.14}). Under a suitable choice of $\omega$ in
(\ref{E:4.15d}) the same proof holds for $l=p>1$ if
$\overline\lambda$ is small enough and hence there exist solutions
of (\ref{v:u})--(\ref{v:n}) with finite time blow-up.
\end{remark}

\section{Blow-up of all nontrivial solutions}\label{all}
In this section we find the conditions which guarantee blow-up in
finite time of all nontrivial solutions of
(\ref{v:u})--(\ref{v:n}).

First we prove that for $p<1$ and $l>1$ no blow-up of all
nontrivial solutions of (\ref{v:u})--(\ref{v:n}) if
 \begin{equation}\label{5.0}
\inf_\Omega c(x,0) > 0.
    \end{equation}
\begin{theorem}
Let  $p<1,\,$ $l>1\,$ and (\ref{5.0}) hold. Then problem
(\ref{v:u})--(\ref{v:n}) has global solutions for small initial
data.
\end{theorem}
\begin{proof}
Thanks to the assumptions of the theorem we have $c(x,t)\geq c_0$
and $k(x,y,t) \leq K$ in $Q_{\tau}$ and $\partial\Omega \times
Q_{\tau},$ respectively, where $c_0,\,$ $K$ and $\tau$ are some
positive constants.

Let $\psi(x)$ be a positive solution of the following problem
    \begin{equation}\label{E:5.0}
\Delta\psi=1,\;x\in\Omega; \;\frac{\partial\psi(x)}{\partial \nu}=
\frac{|\Omega|}{|\partial\Omega|}, \;x\in\partial\Omega.
    \end{equation}
We put
\begin{equation}\label{E:5.01}
b =\inf_\Omega \psi(x)
\end{equation}
and suppose that $f(t)$ is a solution of the following equation
\begin{equation*}
f'(t) = \frac {f(t)}{b} - c_0 b^{p-1} f^p(t).
\end{equation*}
Then $f(t)$ can be written in an explicit form
\begin{equation*}
f(t)= \exp(t/b)\left\{ f^{1-p} (0) -  c_0 b^p \left( 1-
\exp[(p-1)t/b ]\right) \right\}^{1/(1-p)}_+.
\end{equation*}
We assume that
$$
0<f(0) < \left\{ c_0 b^p \left( 1-\exp[(p-1)\tau/b] \right)
\right\}^{1/(1-p)}.
$$
Then $f(t) \equiv 0$ for $t \geq \tau.$

To prove the theorem we construct a supersolution
of~(\ref{v:u})--(\ref{v:n}) in such a form that $v(x,t)=\psi(x)
f(t).$ It is not difficult to check that
\begin{equation}\label{E:5.1a}
v_t - \Delta v + c(x,t)v^p \geq 0
\end{equation}
for $x \in \Omega, \; t>0.$
Now we show that
\begin{equation}\label{E:5.1b}
\frac{\partial v}{\partial\nu} (x,t) \geq \int_{\Omega} k(x,y,t)
v^l(y,t) \, dy \quad x \in \partial\Omega, \; t>0
\end{equation}
for a suitable choice of $f(0).$ Indeed,
\begin{equation*}
\frac{\partial v}{\partial\nu} (x,t) = \frac{|\Omega| }{|\partial\Omega| }
f(t) \geq  \int_{\Omega} k(x,y,t) \psi^l(y) f^l(t) \, dy =
\int_{\Omega} k(x,y,t) v^l(y,t) \, dy, \, x \in \partial\Omega, \; t>0
\end{equation*}
for small values  of $f(0).$ By  (\ref{E:5.1a}), (\ref{E:5.1b}) we
conclude that $v(x,t)$ is  a supersolution
of~(\ref{v:u})--(\ref{v:n})  in $Q_T$ for any $T>0$ if
$$
u_0(x) \leq \psi (x) f (0), \quad x \in \Omega.
$$
Now Theorem~\ref{p:theorem:comp-prins} guarantees the existence of
global solutions of (\ref{v:u})--(\ref{v:n}) for small initial
data.
\end{proof}

The following two statements deal with the case $p=1, \,$ $l>1.$
Let us introduce the notations
\begin{equation*}\label{E:5.1d}
\underline c (t) = \inf\limits_{\Omega} c(x,t), \quad
 \overline k_c (t) = \sup\limits_{\partial\Omega\times\Omega}k(x,y,t) \exp \left\{ -(l-1) \int_0^t  \underline c (\tau) \, d\tau
 \right\},
\end{equation*}
\begin{equation*}\label{E:5.1e}
\underline k_c (x,t) = \inf_{\Omega} k (x,y,t) \exp \left\{ -(l-1)
\int_0^t \overline c (\tau) \, d\tau \right\},
\end{equation*}
where $\overline c (t)$ was defined in (\ref{E:5.1c}).

We prove that any nontrivial solution of
(\ref{v:u})--(\ref{v:n}) blows up in finite time if
\begin{equation}\label{E:5.4}
\int_0^\infty \int_{\partial\Omega} \underline k_c (x,t) \, dS_x
dt = \infty.
\end{equation}
Conversely, problem (\ref{v:u})--(\ref{v:n}) has bounded global
solutions with small initial data, provided that
\begin{equation}\label{E:5.2}
\int_0^\infty \overline k_c (t) \, dt < \infty,
\end{equation}
and there exist positive constants $\alpha,\;t_0$ and $K$ such
that $\alpha>t_0$ and
\begin{equation}\label{E:5.3}
\int\limits_{t-t_0}^t\frac{\overline k_c (\tau)}
{\sqrt{t-\tau}}\,d\tau\leq K\textrm{ for any }t\geq\alpha.
\end{equation}

\begin{theorem}\label{Th11} Let  $p=1, \,$ $l>1$ and
(\ref{E:5.4}) hold. Then any nontrivial solution of
(\ref{v:u})--(\ref{v:n}) blows up at time $t^\star \leq T,$ where
$T$ satisfies the equality
\begin{equation*}
\int_0^T \int_{\partial\Omega} \underline k_c (x,t) \, dS_x dt =
\frac{1}{(l-1)} \left\{ |\Omega| \int_\Omega u_0 (y) \, dy
\right\}^{-(l-1)}.
\end{equation*}
\end{theorem}
\begin{proof}
Let  $v(x,t)$ be a solution of the following problem
  \begin{equation}\label{E:5.7}
v_{t}=\Delta v   \,\,\, \textrm{for} \,\,\,
x \in \Omega, \, t>0,
\end{equation}
  \begin{equation}\label{E:5.8}
\frac{\partial v(x,t)}{\partial \nu} = \underline k_c (x,t)
\int_{\Omega}
v^l (y,t) \, dy \,\,\, \textrm{for}\,\,\, x \in \partial\Omega, \, t>0, \\
\end{equation}
  \begin{equation}\label{E:5.9}
v(x,0)= u_0(x) \,\,\, \textrm{for}\,\,\, x \in
\Omega,
\end{equation}

By a direct computation we can check that
$$
\underline u (x,t) = \exp \left( - \int_0^t \overline c (\tau) \, d\tau \right) v(x,t)
$$
is a subsolution of (\ref{v:u})--(\ref{v:n}) in $Q_T$ for any
$T>0.$ Then  by Theorem~\ref{p:theorem:comp-prins} we have
$$
\underline u (x,t) \leq u(x,t), \, (x,t) \in Q_T
$$
for any $T>0.$ To prove the theorem we show that  any nontrivial
solution of (\ref{E:5.7})--(\ref{E:5.9}) blows up in finite time.
We set
\begin{equation*}
V(t)= \int_\Omega v(x,t) \, dx.
\end{equation*}
Integrating (\ref{E:5.7}) over $\Omega$ and using Green's identity
and Jensen's inequality, we have
  \begin{eqnarray*}
V'(t) &=& \int_\Omega \Delta v(x,t) \,dx  =
 \int_{\partial\Omega} \frac{\partial v(x,t)}{\partial \nu}
\,dS_x = \int_{\partial\Omega} \underline k_c (x,t) dS_x
\int_{\Omega}
v^l(y,t)\,dy \\
&\geq& |\Omega|^{1-l} \int_{\partial\Omega} \underline k_c
(x,t)dS_x V^l(t).
    \end{eqnarray*}
Integrating last inequality, we obtain the desired result due to
(\ref{E:5.4}).
\end{proof}

\begin{theorem}\label{Th12} Let  $p=1,\,$ $l>1$ and
(\ref{E:5.2}), (\ref{E:5.3}) hold. Then problem
(\ref{v:u})--(\ref{v:n}) has bounded global solutions for small
initial data.
\end{theorem}
\begin{proof}
 Let  $w(x,t)$ be a solution of the following problem
\begin{eqnarray}\label{E:5.25}
\left\{
\begin{array}{ll}
w_{t}=\Delta w   \,\,\, \,\,\,& \textrm{for} \,\,\,
x \in \Omega, \, t>0,  \\
\frac{\partial w(x,t)}{\partial \nu} =  \overline k_c (t) \int_{\Omega}
w^l (y,t) \, dy \,\,\,\,\,\,\,\,& \textrm{for}\,\,\, x \in \partial\Omega, \, t>0,  \\
w(x,0)= u_0(x) \,\,\,\,\,\,& \textrm{for}\,\,\, x \in \Omega.
\end{array} \right.
\end{eqnarray}
By a direct computation we can check that
$$
\overline u (x,t) = \exp \left( - \int_0^t  \underline c (\tau) \, d\tau \right) w(x,t)
$$
is a supersolution of (\ref{v:u})--(\ref{v:n}) in $Q_T$ for any
$T>0.$ To prove the theorem we show the existence of global
bounded solutions of (\ref{E:5.25}).  Let us consider the
following auxiliary linear problem
\begin{equation}\label{E:5.26}
\left\{
  \begin{array}{ll}
h_t=\Delta h, \; x\in\Omega, \; t>0 \\
\frac{\partial h(x,t)}{\partial \nu}= \overline
k_c(t),\;x\in\partial\Omega,\;t>0,
    \\
    h(x,0)=h_0(x),\; x\in\Omega.
  \end{array}
\right.
\end{equation}
As it was proved in  \cite{Gladkov_Kavitova2}  any solution of
(\ref{E:5.26}) is a bounded function. Now we construct a
supersolution of (\ref{E:5.25}) in the following form $g(x,t) = a
h(x,t) ,$ where $a$ is some positive constant. It is obvious,
$$
g_t = \Delta g,\;x\in\Omega,\;t>0.
$$
Moreover,
$$
\frac{\partial g(x,t)}{\partial \nu} = a   \overline k_c (t) \geq
a^l   \overline k_c (t) \int_\Omega h^l (y,t) \, dy = \overline
k_c (t) \int_\Omega g^l (y,t) \, dy, \;x\in \partial \Omega,\;t>0,
$$
for small values of $a.$ Then by a comparison principle for
(\ref{E:5.25})
$$
w(x,t) \leq g(x,t), \, (x,t) \in Q_T
$$
for any $T>0$ if  $u_0 (x) \leq a h_0 (x) , \, x \in \Omega.$
\end{proof}
\begin{remark}
By Theorem~\ref{Th11} and Theorem~\ref{Th12} the
condition~(\ref{E:5.2}) is optimal for global existence of
solutions of~(\ref{v:u})--(\ref{v:n}) with $c(x,t) = c(t)$ and
$k(x,y,t) = k(t).$ Arguing in the same way as in the proof of
Lemma~3.3 of \cite{Gladkov_Kavitova2}  it is easy to show
that~(\ref{E:5.3}) is optimal for the existence of nontrivial
bounded global solutions of~(\ref{v:u})--(\ref{v:n}) with $c(x,t)
= c(t)$ and $k(x,y,t)=k(t)$ under the condition
$$
\int_0^\infty  c(t) \, dt < \infty.
$$
 \end{remark}


Now we prove finite time blow-up of all nontrivial solutions of
(\ref{v:u})--(\ref{v:n})  for $l>p>1.$ Let $m_0 = \inf \{
\sup_\Omega \psi (x)\},$ where $\psi (x)$ was defined in
(\ref{E:5.0}). To formulate blow-up result we put
\begin{equation*}
\underline k (t) =
\inf\limits_{\partial\Omega\times\Omega}k(x,y,t)
\end{equation*}
and suppose that
\begin{equation}\label{E:5.27}
c(x,t) \leq c_1 (t),\,\,  c_1 (t) \in C^1([t_0,\infty)), \,\, c_1
(t) > 0 \,\, \textrm{for} \,\, t \geq t_0,
\end{equation}
where $t_0$ is some positive constant,
\begin{equation}\label{E:5.210}
\liminf_{t \to \infty} \frac{c'_1(t)}{c_1(t)} > - \frac {p-1}{m_0}
\end{equation}
and
\begin{equation}\label{E:5.28}
\lim_{t \to \infty} \underline k (t) [c_1(t)]^{(1-l)/(p-1)} =
\infty.
\end{equation}
\begin{theorem}\label{Th13} Let  $l>p>1$ and (\ref{E:5.27}) --
(\ref{E:5.210}) hold. Then any nontrivial solution of
(\ref{v:u})--(\ref{v:n}) blows up in finite time.
\end{theorem}
\begin{proof}
Let $u(x,t)$ be a nontrivial global solution of (\ref{v:u})--(\ref{v:n}). Then by Theorem~\ref{p:theorem:positive}
\begin{equation}\label{E:5.11}
 u(x,t) > 0 \,\, \textrm{for} \,\, x \in \overline\Omega,  \,\, t > 0.
\end{equation}
At first we get an universal lower bound for $u(x,t).$ From
(\ref{E:5.210}) we see that there exists a constant $m$ satisfying
$m > m_0$ and
\begin{equation}\label{E:5.111}
\liminf_{t \to \infty} \frac{c'_1(t)}{c_1(t)} > - \frac {p-1}{m}.
\end{equation}
Let us define $f(t)$ as a solution of the following equation
\begin{equation}\label{E:5.11a}
f'(t) = \frac {f(t)}{m} - m^{p-1} c_1(t)  f^p(t),  \,\, t \geq t_1
\geq t_0,
\end{equation}
Then $f(t)$ can be written in an explicit form
\begin{equation}\label{E:5.12}
f(t)= \exp(t/m)\left\{ [f (t_1)\exp(-t_1/m) ]^{1-p} + (p-1)
m^{p-1} \int_{t_1}^{t} \exp[(p-1)\tau/m ] c_1(\tau) \, d\tau
\right\}^{-\frac {1}{p-1}}.
\end{equation}
We rewrite  (\ref{E:5.12}) as following
\begin{equation}\label{E:5.13}
\left\{ \frac {f(t)}{[c_1(t)]^{-1/(p-1)} }\right\}^{p-1}
\hskip-5pt = \frac{\exp[(p-1)t/m ] c_1(t)}{[f (t_1)\exp(-t_1/m)
]^{1-p} + (p-1) m^{p-1} \int_{t_1}^{t} \exp[(p-1)\tau/m ]
c_1(\tau) \, d\tau}.
\end{equation}
We prove that right hand side $I$ of (\ref{E:5.13}) is bounded
below by some positive constant. The numerator and the denominator
of $I$ tend to infinity as $t \to \infty$ by virtue of
(\ref{E:5.210}). Using (\ref{E:5.111}) we can obtain that
\begin{equation}\label{E:5.15}
\liminf_{t \to \infty} I \geq \liminf_{t \to \infty}
\frac{\exp[(p-1)t/m] \left\{ (p-1)c_1(t)/m + c_1'(t)\right\}}
{(p-1) m^{p-1}  c_1(t) \exp[(p-1)t/m ]} > 0.
\end{equation}
By (\ref{E:5.12}) -- (\ref{E:5.15}) we conclude that
\begin{equation}\label{E:5.16}
f(t) \geq  d_1 [ c_1(t) ]^{-\frac{1}{p-1}}, \,\, t \geq t_1,
\end{equation}
where $d_1 >0.$

Let $\psi (x)$ satisfy (\ref{E:5.0}) and
\begin{equation}\label{E:5.16a}
\sup_\Omega \psi (x) =  m.
\end{equation}
Now we define
\begin{equation}\label{E:5.17}
\underline u (x,t) = \psi (x) f (t)
\end{equation}
and show that $\underline u (x,t)$ is a subsolution of
(\ref{v:u})--(\ref{v:n}) in $\Omega \times (t_1, T)$ under
suitable choice of $t_1$ and $T > t_1.$ Due to (\ref{E:5.0}),
(\ref{E:5.11a}) we have
\begin{equation}\label{E:5.17a}
\underline u_t \leq \Delta \underline u - c(x,t) \underline u^p,
\,\,\, x \in \Omega, \, t>t_1 .
\end{equation}
Using (\ref{E:5.0}), (\ref{E:5.28}), (\ref{E:5.16}), (\ref{E:5.17}) we find that
  \begin{eqnarray}\label{E:5.17b}
\frac{\partial \underline u}{\partial\nu} (x,t) &=& \frac{|\Omega|
}{|\partial\Omega| } f(t) \leq  d_1^{l-1}
[c_1(t)]^{-\frac{l-1}{p-1}} \underline k (t) f(t)
\int_{\Omega} \psi^l(y) \, dy \nonumber \\
&\leq&  \int_{\Omega} k(x,y,t) \underline u^l(y,t) \, dy, \, x \in \partial\Omega, \; t > t_1
 \end{eqnarray}
for large values of $t_1.$ By  (\ref{E:5.11}), (\ref{E:5.16}) --
(\ref{E:5.17b}) and Theorem~\ref{p:theorem:comp-prins}
\begin{equation}\label{E:5.18}
u (x,t) \geq \underline u (x,t) \geq  d_2 [ c_1(t)
]^{-\frac{1}{p-1}}, \,\,\, (x,t) \in \Omega \times (t_1, T)
\end{equation}
for some $d_2 >0$ and any $T > t_1$ if
$$
f(t_1)  \leq \frac{\inf_\Omega u(x,t_1) }{m}.
$$
We set
\begin{equation}\label{E:5.180}
U(t)= \int_\Omega u(x,t) \, dx.
\end{equation}
Integrating (\ref{v:u}) over $\Omega$ and using (\ref{E:5.27}),
(\ref{E:5.28}), (\ref{E:5.18}), (\ref{E:5.180}) and Green's
identity, we have
     \begin{eqnarray}\label{E:5.181}
U'(t) &=& \int_\Omega \left( \Delta u (x,t) - c(x,t) u^p (x,t)
\right) \, dx \geq \int_{\Omega} \left( |\partial\Omega|
\underline k (t) u^{l} (x,t)
- c_1 (t) u^p (x,t) \right) \, dx  \nonumber \\
&\geq& \frac{1}{2} |\partial\Omega| \underline k (t) \int_{\Omega}
u^{l} (x,t) \, dx \geq \frac{1}{2} |\partial\Omega|  d_2^{l-1}
\underline k (t) [c_1 (t)]^{-\frac{l-1}{p-1}} \int_{\Omega} u
(x,t) \, dx \nonumber \\
&=& \xi (t) U (t),
    \end{eqnarray}
where $t \geq t_2, \,$  $t_2$ is large enough and $\lim_{t \to
\infty} \xi (t) \, dt = \infty.$  Integrating (\ref{E:5.181}) over
$(t_2,t)$ we find that
\begin{equation}\label{E:5.182}
U(t)\geq U(t_2) \exp \left( \int_{t_2}^t \xi (\tau) \, d\tau
\right).
\end{equation}

Now we deduce lower bound for $\underline k (t).$ From
(\ref{E:5.111}) we conclude that
\begin{equation}\label{E:5.182a}
c_1(t)\geq c_1(t_3) \exp \left( -\frac {(p-1)t}{m} \right), \,\, t
\geq t_3
\end{equation}
for some $t_3 \geq t_2.$ By (\ref{E:5.28}), (\ref{E:5.182a}) we
have
 \begin{equation}\label{E:5.182c}
\underline k (t) = \gamma_1 (t) [c_1(t)]^{(l-1)/(p-1)} \geq
\gamma_2 (t) \exp  \left( -\frac {(l-1)t}{m} \right) \,\,
\textrm{for} \,\, t \geq t_3,
    \end{equation}
where $\lim_{t \to \infty}  \gamma_i (t) = \infty, \,$ $i=1,2.$

Let us change unknown function
\begin{equation}\label{E:5.183}
w(x,t) =  \exp \left( -\frac{t}{m} \right) u(x,t).
\end{equation}
It is easy to check that $w(x,t)$ is a solution of the following
problem
\begin{eqnarray}
w_t &=& \Delta w- c(x,t)\exp \left( \frac{(p-1)t}{m} \right) w^p -
\frac{1}{m} w,\;x\in\Omega,\;t>0, \label{E:5.184}\\
\frac{\partial w(x,t)}{\partial\nu} &=& \exp \left(
\frac{(l-1)t}{m} \right)
\int_{\Omega}{k(x,y,t)w^l(y,t)}\,dy,\;x\in\partial\Omega,\; t \geq
0, \label{E:5.184a}\\
    u(x,0) &=& u_{0}(x),\; x\in\Omega. \nonumber
\end{eqnarray}

We put
\begin{equation}\label{E:5.185}
W(t)= \int_\Omega w(x,t) \, dx.
\end{equation}
From (\ref{E:5.180}), (\ref{E:5.182}),  (\ref{E:5.183}),
(\ref{E:5.185}) we conclude that
\begin{equation*}\label{E:5.186}
\lim_{t \to \infty} W(t) = \infty.
\end{equation*}

Integrating (\ref{E:5.184}) over $\Omega$ and using
(\ref{E:5.27}), (\ref{E:5.28}), (\ref{E:5.18}), (\ref{E:5.181}),
(\ref{E:5.182c}), (\ref{E:5.184a}), (\ref{E:5.185}), Green's
identity and Jensen's inequality, we have
 \begin{equation*}
W'(t) \geq \sigma (t) W^l (t) - \frac{1}{m} W(t) \,\, \textrm{for}
\,\, t \geq t_3,
    \end{equation*}
where $\lim_{t \to \infty} \sigma (t) = \infty.$ Hence $W(t)$
blows in finite time.
\end{proof}

To prove the optimality of (\ref{E:5.28}) for blow-up of any
nontrivial solution of (\ref{v:u})--(\ref{v:n}) we put
\begin{equation*}
\overline k (t) = \sup\limits_{\partial\Omega\times\Omega}k(x,y,t)
\end{equation*}
and assume that
\begin{equation}\label{E:5.19}
c(x,t) \geq c_2 (t)   \,\, \textrm{for} \,\, t \geq 0,\,\,  c_2
(t) \in C([0,\infty)) \cap C^1([\sigma,\infty)), \,\, c_2 (t) > 0
\,\, \textrm{for} \,\, t \geq \sigma,
\end{equation}
\begin{equation}\label{E:5.20}
\limsup_{t \to \infty} \frac{c'_2(t)}{c_2(t)} \leq 0,
\end{equation}
\begin{equation}\label{E:5.21}
\overline k (t) \leq K_c [c_2(t)]^{(l-1)/(p-1)},\,\, t \geq 0,
\end{equation}
where $\sigma$ and $K_c$ are some positive constants.
\begin{theorem}\label{Th14} Let  $l>p>1$  and (\ref{E:5.19}) --
(\ref{E:5.20}) hold. Then problem (\ref{v:u})--(\ref{v:n}) has
global solutions for small initial data.
\end{theorem}
\begin{proof}
To prove the theorem we construct a supersolution of
(\ref{v:u})--(\ref{v:n}) in $Q_T$ for any $T>0.$ Let us define
$g(t)$ as a positive solution of the following equation
\begin{equation}\label{E:5.22}
g'(t) = \frac {g(t)}{b} - b^{p-1} c_2(t) g^p(t),
\end{equation}
where $b$ was defined in (\ref{E:5.01}). Then $g(t)$ can be
written in an explicit form
\begin{equation}\label{E:5.23}
g(t) = \exp(t/b)\left\{ [g (0)]^{1-p} + (p-1) b^{p-1} \int_{0}^{t}
\exp[(p-1)\tau/b ] c_2(\tau) \, d\tau \right\}^{-1/(p-1)}.
\end{equation}
We rewrite  (\ref{E:5.23}) as following
\begin{equation}\label{E:5.23a}
\left\{ g(t)[c_2(t)]^{1/(p-1)} \right\}^{p-1}  =
\frac{\exp[(p-1)t/b ] c_2(t)}{[g (0)]^{1-p} + (p-1) b^{p-1}
\int_{0}^{t} \exp[(p-1)\tau/b ] c_2(\tau) \, d\tau}.
\end{equation}
Defining the functions
\begin{equation*}
\alpha (t) = \exp[(p-1)t/b ] c_2(t),
\end{equation*}
\begin{equation*}
\beta (t) = [g (0)]^{1-p} + (p-1) b^{p-1} \int_{0}^{t}
\exp[(p-1)\tau/b ] c_2(\tau) \, d\tau
\end{equation*}
and using Cauchy's mean value theorem and (\ref{E:5.20}) for large
values of $a$ we obtain
 \begin{equation}\label{E:5.23b}
\frac{\alpha (t)}{\beta (t)} - \frac{\alpha (a)}{\beta (t)} \leq
 \frac{\alpha (t) - \alpha
(a)}{\beta (t) - \beta (a)} = \frac{\alpha' (\xi)}{\beta' (\xi)} =
\frac{1}{b^p} + \frac{1}{(p-1)b^{p-1}} \frac{c'_2(\xi)}{c_2(\xi)}
\leq \frac{2}{b^p},
    \end{equation}
where $t>a$ and $\xi \in (a,t).$ From (\ref{E:5.23a}),
(\ref{E:5.23b}) we deduce that
\begin{equation}\label{E:5.24}
\left\{ g(t)[c_2(t)]^{1/(p-1)} \right\}^{p-1} \leq
\frac{3}{b^p},\,\, t \geq 0
\end{equation}
for small values of $g(0).$

Now we define
\begin{equation}\label{E:5.24a}
\overline u (x,t) = \psi (x) g (t)
\end{equation}
and show that $\overline u (x,t)$ is a supesolution of
(\ref{v:u})--(\ref{v:n}) in $Q_T$ for any $T>0$ if initial data
are small. By (\ref{E:5.0}), (\ref{E:5.22})  we have
\begin{equation}\label{E:5.24b}
\overline u_t - \Delta\overline u + c(x,t)\overline u^p \geq 0,
\quad x \in \Omega, \, t>0.
\end{equation}
We note that
\begin{equation}\label{E:5.24c}
\lim_{b \to \infty} \frac {m}{b} = 1,
\end{equation}
where $m$ was defined in (\ref{E:5.16a}). Using (\ref{E:5.0}),
(\ref{E:5.21}), (\ref{E:5.24}), (\ref{E:5.24a}), (\ref{E:5.24c})
we find that
\begin{eqnarray}\label{E:5.24d}
\frac{\partial \overline u}{\partial\nu} (x,t) &=& \frac{|\Omega|
}{|\partial\Omega| } g(t) \geq  K_c |\Omega| m^l \left[
\frac{3}{b^{p}} \right]^{\frac{l-1}{p-1}}  g(t) \nonumber \\
&\geq&  K_c \left\{ g(t)[c_2(t)]^{1/(p-1)} \right\}^{l-1} g(t)
\int_{\Omega} \psi^l(y) \, dy
\geq  \overline k (t)  g^l (t) \int_{\Omega} \psi^l(y) \, dy \nonumber \\
&\geq& \int_{\Omega} k(x,y,t) \overline u^l(y,t) \, dy,  \, x \in
\partial\Omega, \; t > 0
 \end{eqnarray}
for large values of $b.$ Thus, by (\ref{E:5.24b}), (\ref{E:5.24d})
and Theorem~\ref{p:theorem:comp-prins} $\overline u (x,t)$ is a
supersolution of (\ref{v:u})--(\ref{v:n}) in $Q_T$ for any $T>0$
if
\begin{equation*}
u_0 (x) \leq g(0) \psi (x).
\end{equation*}
\end{proof}

We shall write $h(x,t)\sim s(t)$ and $z(x,y,t)\sim s(t)$ as
$t\to\infty$ if there exist positive constants $\beta_i, \,
(i=\overline {1,6})$ such that
$$
\beta_1 h(x,t) \leq s(t) \leq \beta_2 h(x,t) \,\, \textrm{for}
\,\, x \in \Omega \,\, \textrm{and} \,\, t \geq \beta_3
$$
and
$$
\beta_4 z(x,y,t) \leq s(t) \leq \beta_5 z(x,y,t) \,\, \textrm{for}
\,\, x \in \partial \Omega, \,\, y \in \Omega \,\,\textrm{and}
\,\, t \geq \beta_6,
$$
respectively.
\begin{remark}
By Theorem~\ref{Th13} and Theorem~\ref{Th14} the condition
(\ref{E:5.28}) is optimal in a certain sense for blow-up in finite
time of any nontrivial solution of~(\ref{v:u})--(\ref{v:n}). In
particular, let $c(x,t)\sim t^\alpha \ln^\beta t \,$ as
$t\to\infty,$ $\alpha \in \mathbb R, \, \beta \in \mathbb R.$ Then
there exist global solutions of~(\ref{v:u})--(\ref{v:n}) for
$k(x,y,t) \leq z(t),$ where $z(t) \sim \{ t^\alpha \ln^\beta t
\}^{ (l-1)/(p-1)}$ as $t\to\infty$ and any nontrivial solution
of~(\ref{v:u})--(\ref{v:n}) blows up in finite time for
$k(x,y,t)\sim \gamma (t) \{ t^\alpha \ln^\beta t \}^{
(l-1)/(p-1)}$ as $t\to\infty$ if $\lim_{t \to \infty}\gamma (t) =
\infty.$
 \end{remark}

\section{Blow-up on the boundary}\label{blow}
\setcounter{equation}{0} \setcounter{theorem}{0}

In this section we show that for problem~(\ref{v:u})--(\ref{v:n})
under certain conditions blow-up  cannot occur at the interior
domain.
\begin{lemma}\label{blow:111}
Let $l> \max \{p, 1 \},$  $\inf\limits_{\partial\Omega\times
Q_T}k(x,y,t)>0$ and the solution $u(x,t)$ of
(\ref{v:u})--(\ref{v:n}) blows up in $t=T$. Then for $t \in [0,T)$
    \begin{equation}\label{blow:inq}
       \int_0^t\int_\Omega u^l(x,\tau)\,dx\,d\tau
\leq s \left( T-t \right)^{-1/(l-1)},\;s>0.
    \end{equation}
\end{lemma}
    \begin{proof}

Integrating (\ref{v:u}) over $Q_t$ and using Green's identity, we
have
   \begin{eqnarray}\label{b1}
     \int_\Omega   u(y,t) \,dy &=&  \int_\Omega u_0(y) \,dy + \int_0^t \int_{\partial\Omega}
     \int_{\Omega}k(\xi,y,\tau) u^l(y,\tau)\,dy dS_\xi d\tau \nonumber\\
   &-&  \int_0^t \int_{\Omega}  c(y,\tau) u^p(y,\tau) \,dy d\tau
   \geq  \int_0^t \int_{\Omega}  \left( k |\partial \Omega| u^l(y,\tau) - C u^p(y,\tau) \right) dy d\tau \nonumber \\
      &\geq& \frac{k|\partial \Omega|}{2} \int_0^t \int_{\Omega} u^l(y,\tau) \,dy d\tau - M,
    \end{eqnarray}
where
$$
k=\inf\limits_{\partial\Omega\times Q_T}k(x,y,t), \;
C=\sup\limits_{ Q_T} c(x,t), \; M = T |\Omega| \left\{ \frac{2
C^{l/p}}{k |\partial \Omega|} \right\}^{\frac{p}{l-p}}.
$$
Applying H$\ddot{o}$lder's inequality, we obtain
\begin{equation}\label{b2}
     \int_\Omega u(y,t) \,dy \leq |\Omega|^{(l-1)/l}   \left\{ \int_\Omega u^l(y,t) \,dy \right\}^{1/l}.
\end{equation}

Let us introduce
$$
J(t) = \int_0^t\int_\Omega u^l(x,\tau)\,dx\,d\tau.
$$
 Now from (\ref{b1}),(\ref{b2}) we have
\begin{equation*}\label{b3}
      \left( J'(t) \right)^{1/l} \geq c_0 J(t) - M_1, \, c_0 >0, M_1 >0.
\end{equation*}
We suppose there exists $t_0 \in (0,T)$ such that $J(t_0) =
2M_1/c_0$ since otherwise (\ref{blow:inq}) holds. Then $J(t) \leq
2M_1/c_0$ for $0 \leq t \leq t_0$ and
\begin{equation}\label{b3}
      J'(t)  \geq \left( \frac{c_0}{2} J(t) \right)^l  \, \textrm{for} \,\,\, t \geq t_0.
\end{equation}
Integrating~(\ref{b3}) over $(t;T),$ we obtain~(\ref{blow:inq}).
\end{proof}

\begin{theorem}
Let the conditions of Lemma~\ref{blow:111} hold. Then for
problem~(\ref{v:u})--(\ref{v:n}) blow-up can occur only on the
boundary.
\end{theorem}
\begin{proof}
In the proof we shall use some arguments of~\cite{Deng_Zhao},
\cite{{Hu_Yin}}. Let $G_N(x,y;t-\tau)$ be the Green function of
the heat equation with homogeneous Neumann boundary condition.
Then we have the representation formula:
    \begin{eqnarray}\label{blow:equat}
        u(x,t)&=&\int_\Omega{G_N(x,y;t)u_0(y)}\,dy - \int_0^t{\int_\Omega{G_N(x,y;t-\tau)c(y,\tau)u^p(y,\tau)}\,dy}\,d\tau\nonumber\\
        &&+\int_0^t{\int_{\partial\Omega}{G_N(x,\xi;t-\tau)\int_{\Omega}{k(\xi,y,\tau)u^l(y,\tau)}\,dy}}\,dS_\xi\,d\tau
    \end{eqnarray}
for $(x,t) \in Q_T.$ We now take an arbitrary
$\Omega'\subset\subset\Omega$ with $\partial\Omega'\in C^2$ such
that $\textrm{dist}(\partial\Omega,\Omega')=\varepsilon>0.$ It is
well known (see, for example,~\cite{Kahane}, \cite{Hu_Yin1}) that
\begin{equation}\label{gl:1G_N}
        G_N(x,y;t-\tau)\geq0,\;x,y\in\Omega,\;0\leq\tau<t<T,
    \end{equation}
    \begin{equation}\label{gl:2G_N}
    \int_{\Omega}{G_N(x,y;t-\tau)}\,dy=1,\;x\in\Omega,\;0\leq\tau<t<T.
    \end{equation}
    \begin{equation}\label{blow:G_N}
    0\leq G_N(x,y;t-\tau)\leq
    c_\varepsilon,\;x\in\Omega',\;y\in\partial\Omega,\;0<\tau<t<T,
    \end{equation}
where $c_\varepsilon$ is a positive constant depending on
$\varepsilon.$ By~(\ref{blow:inq}), (\ref{blow:equat}) --
(\ref{blow:G_N}) we have
    \begin{eqnarray*}
        \sup_{\Omega'} u(x,t) &\leq& \sup_\Omega u_0(x) + c_\varepsilon |\partial \Omega| \sup_{\partial\Omega\times Q_T} k(x,y,t)
       \int_0^t \int_\Omega u^l (y,\tau)\,dy \,d\tau \\
       &\leq& c_1(T-t)^{-1/(l-1)}.
    \end{eqnarray*}
As it is shown in~\cite{Hu_Yin}, there exist a function $f(x)\in
C^2(\overline{\Omega'})$ and positive constant $c_{2}$  such that
    \begin{equation}\label{blow:pr2}
        \Delta f -\frac{l}{l-1}\frac{|\nabla f|^2}{f}\geq-c_{2}\textrm{
        in }\Omega',\;  f(x)>0\textrm{ in }\Omega',\;f(x)=0\textrm{ on
        }\partial\Omega'.
    \end{equation}

Now we compare $u(x,t)$  with
    \begin{equation*}
    w(x,t)=c_{3}\left(f(x)+c_{2}(T-t)\right)^{-1/(l-1)}
    \end{equation*}
in $\Omega' \times (0,T),$ where the positive constant $c_{3}$
will be defined below. By~(\ref{blow:pr2}) for $x\in\Omega'$ and
$t\in[0,T)$ we get
    \begin{equation*}
       w_t - \Delta w + c(x,t) w^p \geq
      \frac{w}{(l-1)[f(x)+c_{2}(T-t)]}\left(c_{2}+\Delta f-\frac{l|\nabla f|^2}{(l-1)[f(x)+c_{2}(T-t)]}\right) \geq
      0.
    \end{equation*}
Choosing $c_{3}$ such  that $c_{3} \ge c^{1/(l-1)}_{2} c_{1}$ and
$w(x,0) \ge u_0(x)$ for $x \in \Omega',$ by comparison principle
we conclude
    \begin{equation*}
    u(x,t)\leq w(x,t)\textrm{ in }{\overline{\Omega'}}\times[0,T).
    \end{equation*}
Hence, $u(x,t)$ cannot blow up in $\Omega'\times[0,T]$. Since
$\Omega'$ is an arbitrary subset of $\Omega$, the proof is
completed.
\end{proof}

From~\cite{ArrietaBernal} it is easy to get the following result.
\begin{theorem}
Let  $ p>1, \, $ $\inf\limits_{Q_T} c(x,t)>0$ and the solution of
(\ref{v:u})--(\ref{v:n}) blows up in finite time. Then blow-up
occurs only on the boundary.
\end{theorem}


\begin{thebibliography}{99}

\bibitem{CL} Carl S, Lakshmikantham V. Generalizad
quasilinearization method for reaction-diffusion equation under
nonlinear and nonlocal flux conditions. J. Math. Anal. Appl.
2002; 271: 182--205.

\bibitem{CY} Cui Z, Yang Z. Roles of weight functions to a nonlinear porous
medium equation with nonlocal source and nonlocal boundary
condition. J. Math. Anal. Appl. 2008; 342: 559--570.

\bibitem{CYZ} Cui Z, Yang Z, Zhang R.  Blow-up of solutions for nonlinear
parabolic equation with nonlocal source and nonlocal boundary
condition. Appl. Math. Comput. 2013; 224: 1--8.

\bibitem{Deng} Deng K.
Comparison principle for some nonlocal problems.
Quart. Appl. Math. 1992; 50: 517--522.

\bibitem{FZ} Fang ZB, Zhang J. Global existence and blow-up properties of solutions
for porous medium equation with nonlinear memory and weighted nonlocal boundary condition.
Z. Angew. Math. Phys. 2015; 66: 67--81.

\bibitem{F} Friedman A. Monotonic decay of solutions of parabolic equations with nonlocal boundary conditions.
Quart. Appl. Math. 1986; 44: 401--407.

\bibitem{GG} Gao Y, Gao W. Existence and blow-up of solutions for a porous medium equation
with nonlocal boundary condition. Appl. Anal. 2011; 90: 799--809.

\bibitem{Gladkov_Kavitova2} Gladkov A, Kavitova T.
Blow-up problem for semilinear heat equation with nonlinear
nonlocal boundary condition, Appl. Anal. 2016; 95: 1974--1988.

\bibitem{GladkovKim} Gladkov A, Kim KI. Blow-up of solutions
for semilinear heat equation with nonlinear nonlocal boundary
condition. J. Math. Anal. Appl. 2008; 338: 264--273.

\bibitem{GN} Gladkov A, Nikitin A. On the existence of global solutions
of a system of semilinear parabolic equations with nonlinear nonlocal boundary conditions.
Differential Equations 2016; 52: 467--482.

\bibitem{GladkovGuedda} Gladkov A, Guedda M.
Blow-up problem for semilinear heat equation with absorption and a
nonlocal boundary condition. Nonlinear Anal. 2011; 74: 4573--4580.

\bibitem{KW} Kong  L,  Wang M.
Global existence and blow-up of solutions to a parabolic system with nonlocal
sources and boundaries. Science in China, Series A. 2007; 50: 1251--1266.

\bibitem{Liu} Liu D, Mu C. Blowup properties for a semilinear reaction-diffusion
system with nonlinear nonlocal boundary conditions. Abstr. Appl.
Anal. 2010; Article ID 148035: 17 pages.

\bibitem{MV} Marras M, Vernier Piro S. Reaction-diffusion problems
under non-local boundary conditions with blow-up solutions.
Journal of Inequalities and Applications. 2014; 2014:167: 11
pages.

\bibitem{Pao} Pao CV. Asimptotic behavior of solutions of reaction-diffusion
equations with nonlocal boundary conditions. J. Comput. Appl.
Math. 1998; 88: 225--238.

\bibitem{WMX} Wang Y, Mu C, Xiang Z. Blowup of solutions to a porous medium equation with nonlocal boundary condition.
Appl. Math. Comput. 2007; 192: 579--585.

\bibitem{YF}  Yang L, Fan C. Global existence and blow-up of solutions to a
degenerate parabolic system with nonlocal sources and nonlocal
boundaries. Monatshefte f\"ur Mathematik. 2014; 174: 493--510.

 \bibitem{YX} Ye Z, Xu X. Global existence and blow-up for a porous medium
system with nonlocal boundary conditions and nonlocal sources.
Nonlinear Anal. 2013; 82: 115--126.

\bibitem{Y} Yin HM.  On a class of parabolic equations with nonlocal
boundary conditions. J. Math. Anal. Appl. 2004; 294: 712--728.

\bibitem{ZK}  Zheng S, Kong L. Roles of weight functions in a nonlinear
nonlocal parabolic system. Nonlinear Anal. 2008; 68: 2406--2416.

\bibitem{Gladkov} Gladkov A.
Initial boundary value problem for a semilinear parabolic equation
with absorption and nonlinear nonlocal boundary condition,
 article is available at https://arxiv.org/abs/1602.05018.

\bibitem{CPE} Cortazar, C., del Pino, M. and Elgueta, M., On the short-time behaviour of the free boundary of a porous medium equation.
Duke J. Math. 1997; 7: 133--149

\bibitem{Deng_Zhao} Deng K, Zhao CL. Blow-up for a parabolic system coupled in an equation and a boundary condition.
Proc. Royal Soc. Edinb. 2001; 131A: 1345--1355.

\bibitem{Hu_Yin} Hu B, Yin HM. The profile near blowup time for solution of the heat equation with a nonlinear boundary condition.
Trans. Amer. Math. Soc.  1994; 346: 117--135.

\bibitem{Kahane} Kahane CS. On the asymptotic behavior of solutions of parabolic
equations. Czechoslovac Math. J. 1983; 33: 262--285.

\bibitem{Hu_Yin1} Hu B, Yin HM. Critical exponents for a system of heat equations coupled in a non-linear boundary condition.
Math. Meth. Appl. Sci.  1996; 19: 1099--1120.

\bibitem{ArrietaBernal} Arrieta J.M., Rodr\'igues-Bernal A.
Localization on the boundary of blow-up for reaction-diffusion equations with nonlinear boundary conditions.
Comm. Partial Differential Equations. 2004; 29: 1127-Ц1148.


\end{thebibliography}
\end{document}